\documentclass[11pt,a4paper,twoside]{amsart}
\usepackage{amssymb,amsmath,amsthm}
\theoremstyle{plain}
\newtheorem{theorem}{Theorem}[section]

\newtheorem{lemma}[theorem]{Lemma}

\newtheorem{proposition}[theorem]{Proposition}
\newtheorem{assumption}[theorem]{Assumption}
\theoremstyle{remark}
\newtheorem{remark}[theorem]{Remark}

\usepackage
[dvipdfm,
   bookmarks=true,
   bookmarksnumbered=false,
   bookmarkstype=toc]
{hyperref}

\numberwithin{equation}{section}

\newcommand{\R}{\mathbb{R}}

\newcommand{\N}{\mathbb{N}}

\newcommand{\F}{\mathcal{F}}

\renewcommand{\Im}{\operatorname{Im}}
\renewcommand{\Re}{\operatorname{Re}}
\newcommand{\I}{\infty}
\newcommand{\abs}[1]{\left\lvert #1\right\rvert}
\newcommand{\norm}[1]{\left\lVert #1\right\rVert}
\newcommand{\Lebn}[2]{\left\lVert #1 \right\rVert_{L^{#2}}}
\newcommand{\Sobn}[2]{\left\lVert #1 \right\rVert_{H^{#2}}}
\newcommand{\hSobn}[2]{\left\lVert #1 \right\rVert_{\dot{H}^{#2}}}

\newcommand{\Jbr}[1]{\left\langle #1 \right\rangle}

\newcommand{\ceil}[1]{\left\lceil#1\right\rceil}
\newcommand{\IN}{\quad\text{in }}

\def\({\left(}
\def\){\right)}
\def\<{\left\langle}
\def\>{\right\rangle}
\def\le{\leqslant}
\def\ge{\geqslant}

\def\d{{\partial}}
\def\l{\lambda}

\newcommand{\eps}{\varepsilon}

\begin{document}
\title[Local existence and WKB approximation for 2D-SP]{
Local existence and WKB approximation of solutions to
Schr\"odinger-Poisson system in the two-dimensional whole space}
\author[S. Masaki]{Satoshi Masaki}
\address{Division of Mathematics\\
Graduate School of Information Sciences\\
Tohoku University\\
Sendai 980-8579, Japan}
\email{masaki@ims.is.tohoku.ac.jp}
\begin{abstract}
We consider the Schr\"odinger-Poisson system in the two-dimensional whole space.
A new formula of solutions to the Poisson equation is used.
Although the potential term solving the Poisson equation may grow at the spatial infinity, 
we show the unique existence of a time-local solution for data in the Sobolev spaces
by an analysis of a quantum hydrodynamical
system via a modified Madelung transform.
This method has been used to justify the WKB approximation of solutions
to several classes of nonlinear Schr\"odinger equation in the semiclassical limit.
\end{abstract}
\maketitle

\section{Introduction}\label{sec:intro}
In this article, we study the following Schr\"odinger-Poisson system
\begin{equation}\label{eq:SP}\tag{SP}
	\left\{
	\begin{aligned}
	&i \d_t u + \frac12 \Delta u = \l Pu, \quad (t,x)\in \R^{1+2},\\
	&-\Delta P = |u|^2, \\
	&\nabla P \in L^\I, \quad
	\nabla P \to 0 \text{ as } |x|\to\I, \quad P(0)=0,\\
	&u(0,x) = u_0(x)
	\end{aligned}
	\right.
\end{equation}
or its semiclassical version
\begin{equation}\label{eq:SPe}\tag{$\text{SP}^\eps$}
	\left\{
	\begin{aligned}
	&i \eps\d_t u^\eps + \frac{\eps^2}2 \Delta u^\eps = \l P^\eps u^\eps, \quad (t,x)\in \R^{1+2},\\
	&-\Delta P^\eps = |u^\eps|^2, \\
	&\nabla P^\eps \in L^\I, \quad
	\nabla P^\eps \to 0 \text{ as } |x|\to\I, \quad P^\eps(0)=0,\\
	&u^\eps(0,x) = u_0^\eps(x),
	\end{aligned}
	\right.
\end{equation}
where $\eps$ is a positive parameter corresponding to the Planck constant
and $\l \in \R\setminus\{0\}$ is a given physical constant.
In this article, we assume that the initial data $u_0$ and $u_0^\eps$
belong to the Soblev space $H^s(\R^2)$ with $s>2$,
and show that \eqref{eq:SP} and \eqref{eq:SPe} have a unique 
time local solution in $C((-T,T);H^s(\R^2))$ 
and that, for the solution $u^\eps$ of \eqref{eq:SPe}, the following WKB type 
approximation is valid in certain topology:
\begin{equation}\label{eq:WKB}
	u^\eps(t,x) = e^{i\frac{\phi(t,x)}{\eps}}
	(a_0 + \eps a_1 + \cdots + \eps^N a_N + o(\eps^N)).
\end{equation}
The Schr\"odinger-Poisson system is one of
the typical example of a nonlinear Schr\"odinger equation with a nonlocal nonlinearity,
and  there are much literature on this system
if space dimension is larger than or equal to three
(see \cite{CazBook} and references therein).

However, in contrast with higher dimensional cases,
the two-dimensional case is less studied.
One of the reason may be that an appropriate meaning of a solution
to the Poisson equation is not so clear.
We briefly recall the treatment of the solution in previous results.
For a given function $b(x) \ge 0$ on $\R^2$, called background, let us replace
the Poisson equation in \eqref{eq:SP} with
\begin{equation}\label{eq:mPoisson}
	-\Delta P = |u|^2 - b.
\end{equation}
One of the most natural solution may be
\[
	P_1:=\F^{-1} \left[ \frac1{|\xi|^2} \F (|u|^2-b)\right],
\]
where $\F$ denotes the Fourier transform.
To admit this solution, the following
 ``neutrality condition'' plays a very important role:
\begin{equation}\label{eq:neutrality}
	2\pi\F(|u(t)|^2-b)(0) = \int_{\R^2} (|u(x,t)|^2 - b(x)) dx \equiv0.
\end{equation}
To handle the strong singularity of $|\xi|^{-2}\F (|u|^2-b)$ at the origin,
this mean-zero assumption is almost essential. 
Indeed, in \cite{AN-MMAS,JW-CPDE,ZhSIAM},
\eqref{eq:neutrality} and
several integrability conditions on $\F(|u|^2-b)$ are assumed to prove that
$P_1$ belongs to some Lebesgue space.

One of the main point in this article is to remove the neutrality condition.
To make this point clear,
we restrict our attention to the zero-background case $b\equiv0$ throughout this article.
Notice that all nontrivial solutions are excluded
when we assume the neutrality condition with $b\equiv0$,
and so that the zero-background case is out of the framework of the previous results.
In this article, the Poisson equation $-\Delta P=|u|^2$ is posed with
the following conditions:
\[
	\nabla P \in L^\I, \quad \nabla P \to 0 \text{ as } |x|\to\I, \quad
	\text{and} \quad  P(0)=0.
\]
Under these conditions, the solution $P$ is given by
\begin{equation}\label{eq:formula}
	P(x) = -\frac1{2\pi}\int_{\R^2} \( \log \frac{|x-y|}{|y|} \) |u(y)|^2 dy
\end{equation}
and unique for a class of $|u|^2$ (see Theorem \ref{thm:2dPoisson}).
It might  seem to be more natural to use the Newtonian potential
\[
	P_2(x) := -\frac1{2\pi}\(\log|x|*|u|^2 \)(x),
\]
as the solution of the Poisson equation.
However, 
$P$ given in \eqref{eq:formula} makes sense merely if $|u|^2 \in L^p(\R^2)$ for some $p \in (1,2)$,
and $P_2$ requires an additional assumption
$\int_{\R^2}(\log|y|)|u(y)|^2dy<\I$ to make sense (see Proposition \ref{prop:convolution}).
Namely, $P$ is well-defined under a weaker assumption.
This is the reason why we consider \eqref{eq:formula}.
It will turn out that the behavior of $P$ at the spatial infinity is not so good;
in general, $|P(x)|=O(\Lebn{u}2^2\log|x|)$ as $|x|\to \I$.
It is worth mentioning that $\nabla P$ never belongs to $L^2(\R^2)$
for all $u\not \equiv 0$ no matter how fast $u$ decays,
say even if $u \in C_0^\I(\R^2)$.
We discuss the Poisson equation in the two-dimensional whole space 
more precisely in Appendix \ref{sec:Poisson}.

In what follows, we consider \eqref{eq:formula} as the solution of the Poisson equation.
With \eqref{eq:formula}, the system \eqref{eq:SP} and
\eqref{eq:SPe} are rewritten as
\begin{equation}\label{eq:SP2}\tag{$\text{SP}^\prime$}
	\left\{
	\begin{aligned}
	&i \d_t u + \frac12 \Delta u = \l P u, \\
	&P =  -\frac{1}{2\pi}\int_{\R^2} \( \log \frac{|x-y|}{|y|} \) |u(y)|^2 dy , \\
	&u(0,x) = u_0(x)
	\end{aligned}
	\right.
\end{equation}
and
\begin{equation}\label{eq:SPe2}\tag{$\text{SP}^{\eps\prime}$}
	\left\{
	\begin{aligned}
	&i \eps\d_t u^\eps + \frac{\eps^2}2 \Delta u^\eps = \l P^\eps u^\eps, \\
	&P^\eps= -\frac{1}{2\pi}\int_{\R^2} \( \log \frac{|x-y|}{|y|} \) |u^\eps(y)|^2 dy , \\
	&u^\eps(0,x) = u_0^\eps(x),
	\end{aligned}
	\right.
\end{equation}
respectively.
The difficulty for solving \eqref{eq:SP2} (and \eqref{eq:SPe2})
lies in the growth $|P(x)|=O(\log|x|)$ at the spatial infinity.
For example, the Duhamel term of the corresponding integral equation
does not necessarily belong to any Lebesgue space because of this growth,
and so we cannot apply directly the usual perturbation argument to the integral equation.

Thus, we apply another method:
We look for the solution of the form
\begin{equation}\label{eq:mMT}
	u^\eps(t,x)= a^\eps(t,x) e^{i\frac{\phi^\eps(t,x)}{\eps}},
\end{equation}
where the ``amplitude part'' $a^\eps$ is complex-valued and
the ``phase part'' $\phi^\eps$ is real-valued.
We are considering only \eqref{eq:SPe2}
because \eqref{eq:SP2} corresponds to the special case $\eps=1$.
Plugging \eqref{eq:mMT} to \eqref{eq:SPe2},
we obtain the following system according to the order of $\eps$:
\begin{equation}\label{eq:S}
	\left\{
	\begin{aligned}
	& \d_t a^\eps + \nabla \phi^\eps \cdot \nabla a^\eps + \frac12 a^\eps \Delta \phi^\eps = i\frac\eps{2}\Delta a^\eps, \\
	& \d_t \phi^\eps + \frac12 |\nabla \phi^\eps|^2 + \l P^\eps = 0, \\
	& P^\eps(t,x) =
	-\frac{1}{2\pi}\int_{\R^2} \( \log \frac{|x-y|}{|y|} \) |a^\eps(t,y)|^2 dy, \\
	& a^\eps(0,x) = A^\eps(x), \quad \phi^\eps(0,x) = \Phi(x)
	\end{aligned}
	\right.
\end{equation}
with $u^\eps_0=A^\eps e^{i\Phi/\eps}$.
Notice that if $(a^\eps,\phi^\eps)$ solves \eqref{eq:S},
then $u^\eps=a^\eps e^{i\phi^\eps/\eps}$ is an exact solution of \eqref{eq:SPe2}.
This decomposition, called a modified Madelung transform,
 is first introduced in \cite{Grenier98}
to justify the WKB approximation \eqref{eq:WKB} of solutions
to a class of nonlinear Schr\"odinger equations,
and extended to several types of nonlinear Schr\"odinger equations in
\cite{AC-ARMA,AC-GP,LL-EJDE} (see also \cite{CaBook,CR-CMP,GLM-TJM,PGEP}).
For the WKB approximation of the Schr\"odinger-Poisson system for other dimensions,
we refer the reader to \cite{LT-MAA} (one dimension) and
\cite{AC-SP,CM-AA,MaLarge} (three dimensions and higher).
In \cite{JW-CPDE,ZhSIAM}, the same limit is treated in the two-dimensional case
by Wigner measures.

\smallbreak
Before stating our result precisely, we make some definitions and notation.
$L^p(\R^2)$ ($p\in[1,\I]$) and $H^{s}(\R^2)$ ($s\ge0$) denote the usual Lebesgue
and Sobolev space, respectively.
We say $f\in H^s_{\mathrm{loc}}(\R^2)$ if any restriction of $f$ on a bounded domain
belongs to $H^s(\R^2)$.
$X^s(\R^2)$ is the Zhidkov space defined by
\begin{align*}
	X^s(\R^2):= \{ f\in L^\I(\R^2) ; \nabla f \in H^{s-1}(\R^2)\}, \\
	\norm{f}_{X^s(\R^2)}:= \norm{f}_{L^\I(\R^2)} + \norm{\nabla f}_{H^{s-1}(\R^2)}
\end{align*}
for $s>1$. This space is introduced in \cite{Zhidkov} (see also \cite{Gallo}).
We sometimes write $L^p=L^p(\R^2)$, $H^s=H^s(\R^2)$, and $X^s=X^s(\R^2)$,
for short.

\subsection{Main result 1}
We first state our result on the unique existence of a local solution to \eqref{eq:SP2}.
\begin{theorem}\label{thm:main1}
Let $u_0 \in H^s(\R^2)$ with $s>2$.
Then, there exist an existence time $T=T(\norm{u_0}_{H^{s-1}(\R^2)})$ and
a time-local solution $u\in C((-T,T);H^s(\R^2))\cap C^1((-T,T);H^{s-2}_{\mathrm{loc}}(\R^2))$ to \eqref{eq:SP2}.
Moreover, this solution is unique in this space and the 
data-to-solution mapping
$u_0 \mapsto u$ is continuous from $H^s(\R^2)$ to $C((-T,T);H^{s-1}(\R^2))$.
Furthermore, the mass is conserved.
\end{theorem}
\begin{remark}
The existence part of Theorem \ref{thm:main1}
holds if $u_0 \in H^s(\R^2)$ ($s>1$).
\end{remark}
\begin{remark}
For the solution $u$ given in Theorem \ref{thm:main1},
$\d_t u$ does not necessarily decays at the spatial infinity.
This is due to the lack of spatial decay of $P$.
Therefore, in general,
$u \in C^1((-T,T);H^{s-2}_{\mathrm{loc}}(\R^2))$
is not replaced by $u \in C^1((-T,T);H^{s-2}(\R^2))$.
Nevertheless, the solution is continuous as a $H^s(\R^2)$-valued function.
\end{remark}
\begin{remark}
The uniqueness of the solution is new in the sense not only  that the result itself is new
but also that we use a new argument.
In our proof, we derive the uniqueness of \eqref{eq:SP2}
from that of the system \eqref{eq:S} via \eqref{eq:mMT}
by showing that every solution of \eqref{eq:SP2} is of the form \eqref{eq:mMT},
at least for small time.
\end{remark}
\begin{remark}
The solution given by Theorem \ref{thm:main1}
does not necessarily have a finite energy. The energy
\[
	E[u](t) = \frac12\Lebn{\nabla u(t)}2^2 - \frac{\l}{4\pi}\iint_{\R^2\times\R^2}
	(\log|x-y|)|u(t,x)|^2|u(t,y)|^2 dx dy
\]
is finite and conserved 
as long as $W(t):=\int_{0}^t \int_{\R^2}(\log|y|)|u(s,y)|^2 dy ds$ is finite.
Indeed, if $W(t)<\I$ then $\widetilde{u}(t)=u(t)e^{i\frac{\l}{2\pi}W(t)}$ solves
\begin{equation}\label{eq:SN}
	i\d_t \widetilde{u} + \frac12 \Delta \widetilde{u} = -\frac{\l}{2\pi}\widetilde{u}
	\int_{\R^2}(\log|x-y|)|\widetilde{u}(y)|^2dy
\end{equation}
and $E[\widetilde{u}](t)$ is conserved.
Then, one sees that 
$E[{u}]=E[\widetilde{u}]$ is also conserved.
The ground state for \eqref{eq:SN} is studied in \cite{Stu2DSN}.
\end{remark}

\subsection{Main result 2}
We next state the WKB approximation of the solution to \eqref{eq:SPe2}.
Suppose that the initial data is of the form
\[
	u^\eps_0(x) = A^\eps(x) e^{i\frac{\Phi}\eps}.
\]
We first make assumption on $u^\eps_0$, that is, on $(A^\eps,\Phi)$.
\begin{assumption}\label{asmp:main2}
Let $N \ge 1$ denote an ``expansion level''.
We assume the following for some $s>3+4N$:
\begin{itemize}
\item $A^\eps \in H^s(\R^2)$ and $\norm{A^\eps}_{H^{s}(\R^2)}$
is uniformly bounded.
Moreover, $A^\eps$ is expanded as
\[
	A^\eps = A_0 + \eps A_1 + \cdots + \eps^{2N} A_{2N} + o(\eps^{2N})
	\IN H^s(\R^2).
\]
Furthermore, there exists a positive number $\alpha \in (0,1]$ such that
\[
	\int_{\R^2} |x|^{\alpha} |A_j|^2 < \I
\]
holds for $j=0,1,2,\dots,N$.
\item $\Phi \in C^1([0,T)\times \R^2)$ with $\nabla \Phi \in X^{s+1}(\R^2)\cap L^p
(\R^2)$ for some $p\in (2,\I)$.
\end{itemize}
\end{assumption}
\begin{theorem}\label{thm:main2}
Let Assumption \ref{asmp:main2} be satisfied for a positive integer $N$.
Then, there exist an existence time $T$ independent of $\eps$ and
a unique time-local solution $u^\eps \in C([0,T);H^s(\R^2))\cap C^1((0,T);H^{s-2}_{\mathrm{loc}}(\R^2))$ to \eqref{eq:SP2}.
Moreover, there exist $\beta_j \in C([0,T); L^2(\R^2))$ ($0\le j < N$) and $\phi_0 \in C([0,T)\times\R^2)$ such that
\[
	u^\eps = e^{i\frac{\phi_0}{\eps}}(\beta_0 + \eps \beta_1 + \cdots
	+ \eps^{N-1} \beta_{N-1} + o(\eps^{N-1})) \IN L^\I([0,T);L^2(\R^2))
\]
for $0<\eps \le 1$.
\end{theorem}
\begin{remark}
Our analysis is still valid in the presence of background.
If we take some nonnegative function $b \in L^1(\R^2)\cap H^s(\R^2)$
and if we change the Poisson equation into $-\Delta P^\eps=|u^\eps|^2-b$,
then the same result as in Theorem \ref{thm:main2} holds without the neutrality
condition $\int_{\R^2}(|u^\eps|^2-b)dx=0$.
\end{remark}

\smallbreak
The rest of this article is organized as follows:
In Section \ref{sec:preliminary}, we first show the existence of a solution to 
the system \eqref{eq:S} on which
our main theorems are all based.
Then, we prove Theorems \ref{thm:main1} and \ref{thm:main2}
in Sections \ref{sec:proof1} and \ref{sec:proof2}, respectively.
In Appendix \ref{sec:Poisson}, we summarize results on
the Poisson equation in the two-dimensional whole space.
They play  important roles in our argument.

\section{Preliminary result}\label{sec:preliminary}
In this section, we establish an existence result on the system
\begin{equation}\tag{\ref{eq:S}}
	\left\{
	\begin{aligned}
	& \d_t a^\eps + \nabla \phi^\eps \cdot \nabla a^\eps + \frac12 a^\eps \Delta \phi^\eps = i\frac\eps{2}\Delta a^\eps, \\
	& \d_t \phi^\eps + \frac12 |\nabla \phi^\eps|^2 + \l P^\eps = 0, \\
	& P^\eps (t,x)=
	-\frac{1}{2\pi}\int_{\R^2} \( \log \frac{|x-y|}{|y|} \) |a^\eps(t,y)|^2 dy, \\
	& a^\eps(0,x) = A^\eps(x), \quad \phi^\eps(0,x) = \Phi(x).
	\end{aligned}
	\right.
\end{equation}

\begin{assumption}\label{asmp:base}
We assume the following for some $s>2$:
\begin{itemize}
\item $A^\eps \in H^s(\R^2)$  and $\norm{A^\eps}_{H^s(\R^2)}$
is uniformly bounded.
\item $\Phi \in C^1([0,T)\times \R^2)$ with $\nabla \Phi \in X^{s+1}(\R^2)\cap L^p
(\R^2)$ for some $p\in (2,\I)$.
\end{itemize}
\end{assumption}
\begin{theorem}\label{thm:base}
Let Assumption \ref{asmp:base} be satisfied.
Then, 
there exist $T>0$ independent of $\eps$ and a unique solution
\begin{align*}
	&a^\eps \in C([0,T); H^s(\R^2)) \cap C^1((0,T);H^{s-2}(\R^2)) ,\\
	&\phi^\eps \in C^1([0,T)\times \R^2)
\end{align*}
of \eqref{eq:S}.
Moreover, $\norm{a^\eps}_{H^s(\R^2)}$ and
$\norm{\nabla \phi^\eps}_{X^{s+1}(\R^2)\cap L^p (\R^2)}$ are uniformly bounded,
and the data-to-solution mapping $(A^\eps,\nabla \Phi) \mapsto (a^\eps,\nabla \phi^\eps)$ is continuous from $H^{s}(\R^2) \times (X^{s+1}(\R^2)\cap L^p(\R^2))$ to
$C([0,T);H^{s-1}(\R^2)\times (X^{s}(\R^2)\cap L^p(\R^2)))$.
The mass $\norm{a^\eps(t)}_{L^2(\R^2)}$ is conserved and
it holds that
\[
	\limsup_{|x|\to\I}\frac{|\phi^\eps(t,x)-\Phi(x)|}{\log|x|}\le \frac{t|\l|}{2\pi}\Lebn{A^\eps}2^2.
\]
Furthermore, if $s>3$ and $A_0 := \lim_{\eps\to0} A^\eps$ exists in $H^s(\R^2)$ then
the following properties hold:
\begin{itemize}
\item $(a_0,\phi_0):=(a^\eps,\phi^\eps)_{\eps=0}$ exists in the same class
and solves
\begin{equation}\label{eq:S0}
	\left\{
	\begin{aligned}
	& \d_t a_0 + \nabla \phi_0 \cdot \nabla a_0 + \frac12 a_0 \Delta \phi_0 = 0, \\
	& \d_t \phi_0 + \frac12 |\nabla \phi_0|^2 + \l P_0 = 0, \\
	& a_0(0,x) = A_0(x), \quad \phi_0(0,x) = \Phi(x),
	\end{aligned}
	\right.
\end{equation}
where $P_0$ is the Poisson term defined from $a_0$ by \eqref{eq:formula},
and $(a^\eps,\nabla \phi^\eps)$ converges to $(a_0,\nabla \phi_0)$ as $\eps\to0$
in the $C([0,T]; H^{s-2}\times (X^{s-2} \cap L^{2+}))$ topology.
\item For any bounded domain $\Omega \subset \R^2$, it holds that
$\norm{\phi^\eps - \phi_0}_{L^\I([0,T]\times \Omega)}\to 0$ as $\eps \to 0$.
\end{itemize}
\end{theorem}

\begin{proof}
Put $v^\eps:= \nabla \phi^\eps$ and consider
\begin{equation}\label{eq:SHS}
\left\{
\begin{aligned}
&\d_t a^\eps + v^\eps \cdot \nabla a^\eps + \frac12 a^\eps \nabla \cdot v^\eps
 = i\frac\eps2 \Delta a^\eps , \quad a^\eps(0,x)=A_0^\eps(x);\\
&\d_t v^\eps + (v^\eps \cdot \nabla )v^\eps + \l\nabla  P^\eps=0, \quad v^\eps(0,x)=\nabla \Phi_0(x);\\
& P^\eps (t,x)=
	-\frac{1}{2\pi}\int_{\R^2} \( \log \frac{|x-y|}{|y|} \) |a^\eps(t,y)|^2 dy, \\
&|\nabla P^\eps| \in L^\I (\R^2), \quad|\nabla P^\eps|\to0 \text{ as } |x|\to\I, \quad
P^\eps(0)=0.
\end{aligned}
\right.
\end{equation}
Let us prove that this system has a unique solution.
To show the existence, it suffices to establish an energy estimate
\begin{equation}\label{eq:energyest}
	\sup_{t\in[0,T]}E(t)\le C E(0)
\end{equation}
for some positive constants $T$ and $C$ independent of $\eps$, where
\[
	E(t):= \Sobn{a^\eps(t)}s^2 + \Lebn{v^\eps(t)}{p}^2+\Lebn{v^\eps(t)}{\I}^2
	 + \Sobn{\nabla v^\eps(t)}s^2.
\]
However, we omit the details of this part
because this part is easier than the following uniqueness part
and is essentially the same as in \cite{MaLarge}.
Then, the standard argument shows $(a^\eps,v^\eps)
\in C([0,T);H^s \times X^{s+1}\cap L^p)$ exists (see \cite{AC-SP,CM-AA}).
It follows from the first line of \eqref{eq:SHS} that $a^\eps \in C^1((0,T);H^{s-2})$.
Since $E(0)$ is bounded uniformly in $\eps$,
we see that the solution is also uniformly bounded.
The conservation of $\Lebn{a^\eps(t)}2$ is also obtained from
the first line of \eqref{eq:SHS} by a standard argument.
\smallbreak

Now, we proceed to the proof of the uniqueness of \eqref{eq:SHS}.
Assume $3>s>2$ and put $\sigma = s-1$.
Let $(a^\eps_1, \phi^\eps_1)$ and $(a^\eps_2, \phi^\eps_2)$ be two solutions
with data $(A^\eps,\Phi)$.
One sees that $(b^\eps,w^\eps):=(a_1^\eps-a_2^\eps,v_1^\eps-v_2^\eps)$ solves
\begin{equation}\label{eq:dSHS}
	\left\{
	\begin{aligned}
	& \d_t b^\eps + v_1^\eps  \cdot \nabla b^\eps + w^\eps  \cdot \nabla a_2^\eps
	+\frac12 a_1^\eps \nabla \cdot w^\eps
	+ \frac12 b^\eps \nabla \cdot v_2^\eps = \frac{i\eps}2 \Delta b^\eps, \\
	& \d_t w^\eps + v_1^\eps \cdot \nabla w^\eps + w^\eps \cdot \nabla v_2^\eps + \l (\nabla P_1^\eps - \nabla P_2^\eps) =0, \\
	& b^\eps(0)=0, \quad w^\eps(0)=0,
	\end{aligned}
	\right.
\end{equation}
where $P^\eps_1$ and $P_2^\eps$ are the Poisson terms determined from
$a_1^\eps$ and $a_2^\eps$, respectively.
We now estimate $H^{\sigma}$-norm of $b^\eps$
by the commutator estimate (see \cite{LaJFA})
\begin{align*}
	&|\Re\Jbr{\Lambda^\sigma(v_1^\eps\cdot \nabla b^\eps), \Lambda^\sigma b^\eps}| \\
	&{}\le \frac12 \Lebn{\nabla v_1^\eps}\I \Sobn{b^\eps}{\sigma}^2
	+ |\Re\Jbr{[\Lambda^\sigma,v_1^\eps\cdot \nabla] b^\eps, \Lambda^\sigma b^\eps}| \\
	&{}\le C \Lebn{\nabla v_1^\eps}\I \Sobn{b^\eps}{\sigma}^2 \\
	&\quad +C(\Lebn{\nabla v_1^\eps}\I \Sobn{\nabla b^\eps}{\sigma-1} + \norm{\nabla v_1^\eps}_{W^{\sigma,r}} \Lebn{\nabla b^\eps}{q} )\Sobn{b}{\sigma},
\end{align*}
where $q=2/(2-\sigma)$ and $r=2/(\sigma-1)$.
Notice that $L^q \hookrightarrow \dot{H}^{\sigma-1}$ and
$L^r \hookrightarrow \dot{H}^{2-\sigma}$.
Therefore the right hand side is bounded by
\[
	C(\Lebn{\nabla v_1^\eps}\I  + \Sobn{\nabla v_1^\eps}2 )\Sobn{b^\eps}{\sigma}^2.
\]
One sees that
\begin{multline*}
	|\Re\Jbr{\Lambda^\sigma(w^\eps\cdot \nabla a_2^\eps), \Lambda^\sigma b^\eps}| \\
	\le C(\Lebn{w^\eps}{\I}\Sobn{a_2^\eps}{s} + \Sobn{\nabla  w^\eps}{\sigma-1}
	\Lebn{\nabla a_2^\eps}{\I})\Sobn{b^\eps}{\sigma},
\end{multline*}
and, similarly, that
\begin{align*}
	|\Re\Jbr{\Lambda^\sigma(a_1^\eps \nabla \cdot w^\eps), \Lambda^\sigma b^\eps}|
	&{}\le C\Sobn{a_1^\eps}{\sigma}\Sobn{\nabla  w^\eps}{\sigma}\Sobn{b^\eps}{\sigma},\\
	|\Re\Jbr{\Lambda^\sigma(b^\eps \nabla \cdot v_2^\eps), \Lambda^\sigma b^\eps}|
	&{}\le C\Sobn{b^\eps}{\sigma}^2\Sobn{\nabla  v_2^\eps}{\sigma}.
\end{align*}
We summarize above estimates to end up with
\[
	\frac{d}{dt} \Sobn{b^\eps}{\sigma}^2 \le C (\Sobn{b^\eps}{\sigma}^2 + \Lebn{w}{p}^2
	+\Sobn{\nabla w^\eps}{\sigma}^2).
\]
Let us proceed to the estimate on $w$.
From the second line of \eqref{eq:dSHS},
\[
	\frac{d}{dt}|w^\eps(t,x)|\le |\d_t w^\eps(t,x)| \le  \Lebn{v_1^\eps \cdot \nabla w^\eps + w^\eps \cdot \nabla v_2^\eps + \l (\nabla P_1^\eps - \nabla P_2^\eps)}{\I},
\]
which yields
\[
	\frac{d}{dt} \Lebn{w^\eps}\I \le C(\Sobn{\nabla w^\eps}{\sigma} + \Lebn{w^\eps}{\I}
	+ \Sobn{b^\eps}\sigma).
\]
Similarly, operating $\nabla$ to the second line of \eqref{eq:dSHS}, we obtain
\[
	\d_t\nabla w^\eps +\nabla( v_1^\eps \cdot \nabla w^\eps )+ \nabla(w^\eps \cdot \nabla v_2^\eps) + \l \nabla^2( P_1^\eps - P_2^\eps) =0.
\]
In the essentially same way as in the estimate on $b^\eps$, we obtain
\begin{align*}
	&|\Jbr{\Lambda^\sigma \nabla(v_1^\eps \cdot \nabla w^\eps) ,\Lambda^\sigma \nabla w^\eps}| \\
	&{}\le C\Sobn{\nabla w^\eps}{\sigma}^2\Lebn{\nabla v_1^\eps}\I \\
	&\quad + C (\Lebn{\nabla v_1^\eps}\I \Sobn{\nabla w^\eps}{\sigma} + \Sobn{\nabla v_1^\eps}{s}\Lebn{\nabla w^\eps}\I)
\end{align*}
and
\[
	\Lebn{\Lambda^\sigma \nabla(w^\eps \cdot \nabla v_2^\eps)}{2}
	\le C (\Sobn{\nabla w^\eps}{\sigma} \Lebn{\nabla v_2^\eps}\I
	+ \Lebn{w^\eps}{\I}\Sobn{\nabla v_2^\eps}{s}).
\]
By the use of $L^2$-boundedness of the Riesz transform,
\[
	\Sobn{\nabla^2( P_1^\eps - P_2^\eps)}{\sigma} \le C\Sobn{|a_1^\eps|^2-|a_2^\eps|^2}{\sigma}
	\le C \Sobn{b^\eps}{\sigma}.
\]
Therefore,
\[
	\frac{d}{dt} \Sobn{\nabla w^\eps}{\sigma} \le C(\Sobn{\nabla w^\eps}{\sigma} + \Lebn{w^\eps}{\I}
	+ \Sobn{b^\eps}\sigma).
\]
Hence, we apply Gronwall's lemma to the inequality
\[
	\frac{d}{dt} (\Sobn{b^\eps}\sigma+ \Lebn{w^\eps}{\I}
	+ \Sobn{\nabla w^\eps}{\sigma})
	\le C(\Sobn{b^\eps}\sigma+ \Lebn{w^\eps}{\I}
	+ \Sobn{\nabla w^\eps}{\sigma})
\]
to conclude $b\equiv0$ and $ w\equiv 0$,
which shows the uniqueness.
The continuous dependence on the initial data is proven in the essentially same way.
\smallbreak

We now suppose $s>3$ and that $A_0:=\lim_{\eps\to0} A^\eps \in H^s(\R^2)$ exists,
and prove the convergence of $(a^\eps,v^\eps)$ as $\eps\to0$.
So far, we do not use the property $\eps\neq0$.
Hence, repeating the above argument, we see that the system \eqref{eq:SHS}
with $\eps=0$ has a unique solution $(a_0,v_0)$ in the same class. 
One sees that $(b^\eps,w^\eps):=(a^\eps-a_0,v^\eps-v_0)$ solves
\[
	\left\{
	\begin{aligned}
	& \d_t b^\eps + v^\eps \cdot \nabla b^\eps + w^\eps \cdot \nabla a_0
	+\frac12 a^\eps \nabla \cdot w^\eps
	+ \frac12 b^\eps \nabla \cdot v_0 = \frac{i\eps}2 \Delta b^\eps + \frac{i\eps}2 \Delta a_0, \\
	& \d_t w^\eps + v^\eps \cdot \nabla w^\eps + w^\eps \cdot \nabla v_0 + \l (\nabla P_{\eps} - \nabla P_0) =0, \\
	& b^\eps(0)=A^\eps - A_0, \quad w^\eps(0)=0,
	\end{aligned}
	\right.
\]
where $P_\eps$ and $P_0$ are the Poisson terms determined from
$a^\eps$ and $a_0$, respectively.
The estimates we used for the proof of uniqueness give us the convergence result.
The difference is just that the term $\Delta a_0$ produces two-derivative loss.
We note that
\begin{equation}\label{eq:vconverge}
	\norm{v^\eps - v_0}_{L^\I([0,T);L^r)}\to 0, \quad \text{ as } \eps \to 0
\end{equation}
is true not only for $r\ge p$ but also for all $r>2$.
Indeed, if \eqref{eq:vconverge} holds for some $r=2r_0>4$, the estimate
\begin{multline*}
	\norm{v^\eps - v_0}_{L^\I_TL^{\frac{r}2}}
	\le T\norm{v^\eps - v_0}_{L^\I_TL^{r}}\norm{\nabla v^\eps }_{L^\I_TL^{r}} \\
	+ T\norm{v_0 }_{L^\I_TL^{\I}}\norm{\nabla (v^\eps - v_0)}_{L^\I_TL^{\frac{r}2}} +  T\norm{\nabla (P^{\eps}-P_0)}_{L^\I_TL^{\frac{r}2}}
\end{multline*}
shows that \eqref{eq:vconverge} holds also for $r/2=r_0>2$,
where we write $L_T^\I L^r=L^\I([0,T);L^r)$, for short. 
Since we have already known that \eqref{eq:vconverge} holds for all $r \in [p,\I]$,
the $k$-time use of this argument proves \eqref{eq:vconverge} for
$r > \max(p/2^k,2)$.
The lower bound $r>2$ comes from the estimate on the Poisson term
(see Remark \ref{rmk:nPL2}).
\smallbreak

We construct $\phi^\eps$ as
\[
	\phi^\eps(t):= \Phi - \int_0^t \( \frac12 |v^\eps(s)|^2  + \l P^\eps(s) \)ds.
\]
Note that $\nabla \times v^\eps(t) \equiv 0$ holds at $t=0$, and so for all $t\in [0,T]$.
One verifies that $\nabla |v^\eps|^2 = 2(v^\eps \cdot \nabla) v^\eps$
and so that $(a^\eps, \nabla \phi^\eps)$ solves \eqref{eq:SHS}.
Then, $\nabla \phi^\eps=v^\eps$ by uniqueness.
Recall that $v^\eps$ decays at spatial infinity.
Hence, we deduce from \eqref{eq:log} that
\[
	\limsup_{|x|\to\I} \frac{|\phi^\eps(t,x)-\Phi(x)|}{\log|x|} \le |\l|\int_0^t \frac{\Lebn{a^\eps(s)}2^2}{2\pi} ds = \frac{t|\l|}{2\pi}\Lebn{A^\eps}2^2,
\]
where we have used the fact that $\Lebn{a^\eps(t)}2$ is conserved.
It is clear from above representation of $\phi^\eps$ to see that, for any bounded domain
$\Omega \subset \R^2$, $\phi^\eps -\Phi$ is bounded in $L^\I ([0,T]\times \Omega)$ uniformly in $\eps$.
Moreover,
\begin{align*}
	\norm{\phi^\eps-\phi_0}_{L^\I([0,T)\times \Omega)}\le{}&
	 \frac{T}2 \norm{|v^\eps|^2-|v_0|^2}_{L^\I([0,T)\times \Omega)} \\&{} + |\l|T \norm{P^\eps-P_0}_{L^\I([0,T)\times \Omega)} \\
	 \le{}& C_\Omega (\norm{v^\eps-v_0}_{L^\I([0,T)\times \Omega)}
	+ \norm{a^\eps-a_0}_{L^\I([0,T)\times \Omega)}) \\
	\to{}&0
\end{align*}
as $\eps\to0$.
\end{proof}

\section{Proof of Theorem \ref{thm:main1}}\label{sec:proof1}
Letting $\eps=1$, $A^\eps=u_0$, and $\Phi\equiv0$ in Theorem \ref{thm:base},
we see that if $u_0 \in H^s$ ($s>2$) then the system
\begin{equation}\label{eq:S1}
	\left\{
	\begin{aligned}
	& \d_t a + \nabla \phi \cdot \nabla a + \frac12 a \Delta \phi = i\frac1{2}\Delta a, \\
	& \d_t \phi + \frac12 |\nabla \phi|^2 + \l P = 0, \\
	& P(t,x) =
	-\frac{1}{2\pi}\int_{\R^2} \( \log \frac{|x-y|}{|y|} \) |a(y)|^2 dy, \\
	& a(0,x) = u_0(x), \quad \phi(0,x) = 0,
	\end{aligned}
	\right.
\end{equation}
has a unique solution $(a,\phi)$ satisfying
\begin{align*}
	a &{}\in C([0,T); H^s(\R^2)) \cap C^1((0,T);H^{s-2}(\R^2)), \\
	\phi &{}\in C^1([0,T)\times \R^2), \quad
	\nabla \phi \in L^\I ([0,T); X^{s+1} (\R^2)\cap L^{2+}(\R^2)).
\end{align*}
Moreover, $\Lebn{a}2$ is conserved, $\phi(x)=O(\log|x|)$ as $|x|\to\I$,
and the mapping $u_0\mapsto (a,\nabla \phi)$ is
continuous from $H^s$ to $C([0,T);H^{s-1}(\R^2)) \times C([0,T);X^s(\R^2)\cap L^{2+}(\R^2))$. 
We begin our discussion from this point.
Since the system \eqref{eq:SP2} is time-reversible
we only consider for positive time in what follows.
\subsection{Existence}
One easily verifies that $u=ae^{i\phi}$ solves \eqref{eq:SP2} 
 in the $L^2$ sense
because the first line and the second line of \eqref{eq:S1} are satisfied
in the $L^2$ sense and in the classical sense, respectively.
Our fist goal is to show that this $u$ belongs to $C([0,T);H^s(\R^2))$.
It immediately follows from the following lemma that $u\in L^\I([0,T);H^s(\R^2))$.
\begin{lemma}\label{lem:HsofWKB}
For any $s>1$,
\begin{equation}\label{eq:HsofWKB}
	\Sobn{ae^{i\phi}}{s} \le C \Sobn{a}{s}(1+
	\Sobn{\nabla^2 \phi}{\max(s-2,0)}) (1+ \Lebn{\nabla\phi}{\I}^{\ceil{s}}),
\end{equation}
where $\ceil{s}$ denotes the minimum integer larger than or equal to $s$.
\end{lemma}
\begin{remark}
A similar estimate can be found in \cite{WdJFA}.
The good point in our estimate is that
we do not need any bound on $\phi$ itself.
\end{remark}
\begin{proof}
We first consider the case $1<s<2$.
Note that
\[
	\Sobn{ae^{i\phi}}{s}
	\sim \Lebn{ae^{i\phi}}{2} + \hSobn{ae^{i\phi}}{s}.
\]
The first term of the right hand side is nothing but $\norm{a}_{L^2}$.
For $1<s<2$,
\begin{multline*}
	\hSobn{ae^{i\phi}}{s} \\ 
	\sim\( \int_0^\I \( t^{-s} \sup_{|y|\le t} \norm{\delta_y(ae^{i\phi}) - 2ae^{i\phi} +\delta_{-y}(ae^{i\phi})}_{L^2(\R^2)} \)^2 \frac{dt}{t} \)^{\frac12}
\end{multline*}
is well known, where $\delta_y$ denotes the shift operator, $(\delta_y f)(x):=f(x+y)$
(see \cite[Theorem 6.3.1]{BLBook}).
An elementary calculation shows
\begin{align*}
	&\Lebn{\delta_y(ae^{i\phi}) - 2ae^{i\phi} +\delta_{-y}(ae^{i\phi})}2 \\
	&{} \le 
	\Lebn{(\delta_y a - 2a +\delta_{-y}a)\delta_y e^{i\phi}}2
	+ \Lebn{(a - \delta_{-y}a)(\delta_y e^{i\phi}-\delta_{-y} e^{i\phi})}2 \\
	&\quad {} + \Lebn{a(\delta_y e^{i\phi}- 2e^{i\phi} +\delta_{-y}e^{i\phi})}2.
\end{align*}
The first two terms in the right hand side are estimated as
\begin{align*}
	\Lebn{(\delta_y a - 2a +\delta_{-y}a)\delta_y e^{i\phi}}2 \le \Lebn{\delta_y a - 2a +\delta_{-y}a}2
\end{align*}
and
\[
	\Lebn{(a - \delta_{-y}a)(\delta_y e^{i\phi}-\delta_{-y} e^{i\phi})}2 
	\le \Lebn{a - \delta_{-y}a}{2} \min (2, 2|y| \Lebn{\nabla \phi}{\I}) ,
\]
respectively. The third satisfies
\begin{multline*}
	\Lebn{a(\delta_y e^{i\phi}- 2e^{i\phi} +\delta_{-y}e^{i\phi})}2 \\
	\le 
	\min(4 \Lebn{a}{2}, 
	2|y|^2  \Lebn{\nabla \phi}{\I}^2 \Lebn{a}{2} + 2|y|^2 \Lebn{\nabla^2 \phi}{2}  \Lebn{a}{\I}).
\end{multline*}
Combining all these estimates, we conclude that
\[
	\Sobn{ae^{i\phi}}{s} \le
	C \Sobn{a}s (1+ \Lebn{\nabla^2 \phi}{2})(1+ \Lebn{\nabla \phi}{\I}^2),
\]
which proves \eqref{eq:HsofWKB} for $1<s<2$.

If $s=2$ then \eqref{eq:HsofWKB} is obvious by the H\"older inequality.

Let us proceed the case $s>2$. We prove by induction.
Take some integer $k\ge1$ and assume that
\eqref{eq:HsofWKB} is true for $k<s \le k+1$. Then,
\[
	\Sobn{ae^{i\phi}}{s+1}
	\sim \Lebn{ae^{i\phi}}{2} + \Lebn{|\nabla|^{s+1}ae^{i\phi}}{2}
\]
holds and the tame estimate gives us
\begin{align*}
	\Lebn{|\nabla|^{s+1}ae^{i\phi}}{2}
	\le{}& C \Lebn{|\nabla|^{s}(\nabla(ae^{i\phi}))}{2} \\
	\le{}& C \(\Sobn{(\nabla a)e^{i\phi}}{s}
	+ \Sobn{a e^{i\phi}\nabla \phi}{s}\) \\
	\le{}&C \Sobn{(\nabla a)e^{i\phi}}{s} 
	+ C\Sobn{a e^{i\phi}}{s} \Lebn{\nabla \phi}{\I}\\
	&{}+ C\Lebn{a e^{i\phi}}{\I}\Sobn{\nabla^2 \phi}{s-1}.
\end{align*}
By assumption of the induction,  we obtain
\[
	\Sobn{(\nabla a)e^{i\phi}}{s}
	\le C\Sobn{\nabla a}{s}(1+
	\Sobn{\nabla^2 \phi}{\max(s-2,0)}) (1+ \Lebn{\nabla\phi}{\I}^{\ceil{s}})
\]
and
\begin{multline*}
	\Sobn{a e^{i\phi}}{s} \Lebn{\nabla \phi}{\I}\\
	\le C\Sobn{a}{s}(1+
	\Sobn{\nabla^2 \phi}{\max(s-2,0)}) (1+ \Lebn{\nabla\phi}{\I}^{\ceil{s}}) \Lebn{\nabla \phi}{\I}.
\end{multline*}
Since $s>k\ge 1$, the Sobolev embedding $H^s(\R^2)
\hookrightarrow L^\I(\R^2)$ implies
\[
	\Lebn{a e^{i\phi}}{\I}\Sobn{\nabla^2 \phi}{s-1}
	\le C \Sobn{a}{s} \Sobn{\nabla^2 \phi}{s-1}.
\]
Together with these estimates, we conclude that
\[
	\Sobn{ae^{i\phi}}{s+1}\le C \Sobn{a}{s+1}(1+
	\Sobn{\nabla^2 \phi}{s-1}) (1+ \Lebn{\nabla\phi}{\I}^{\ceil{s}+1}),
\]
which shows that \eqref{eq:HsofWKB} is true for $k+1<s+1\le k+2$.
\end{proof}
\begin{remark}
The following estimate can be established in the same way; for $0<s<1$ and $1\le p,q \le \I$,
\[
	\norm{ae^{i\phi}}_{B^s_{p,q}(\R^2)} \le C\( \norm{a}_{B^s_{p,q}(\R^2)} +
	\norm{\nabla \phi}_{L^\I(\R^2)} \norm{a}_{L^p(\R^2)}\),
\]
where $B^s_{p,q}(\R^2)$ denotes the Besov space.
\end{remark}

\subsection{Continuity}
The following lemma confirms that $u$ is continuous in time
as $H^s(\R^2)$-valued function.
\begin{lemma}\label{lem:ContHsofWKB}
Let $s>0$. 
Assume $\Sobn{a}{s}$ is bounded.
For any $\eps>0$ there exists $\delta>0$ such that
if $\Lebn{\nabla\phi}{\I} + \Sobn{\nabla^2\phi}{\max(s-2,0)}+|\phi(0)| < \delta$ then
$\Sobn{a (e^{i\phi}-1)}{s} \le \eps$.
\end{lemma}
Indeed, an elementary calculation shows that
\begin{align*}
	\Sobn{a_1 e^{i\phi_1} -a_2 e^{i\phi_2}}{s}
	\le{}& \Sobn{(a_1-a_2)e^{i\phi_1}}{s}
	+ \Sobn{a_2e^{i\phi_2}(e^{i(\phi_1-\phi_2)}-1)}{s}.
\end{align*}
We now fix $t\in (0,T)$ and take $(a_1,\phi_1)=(a(t+h),\phi(t+h))$
and $(a_2,\phi_2)=(a(t),\phi(t))$.
Then, as $h\to0$, the first term tends to zero
because of the previous Lemma \ref{lem:HsofWKB},
and so does the second term because of this lemma.
Namely, we obtain the desired continuity.
The continuous dependence of $u$ on the data $u_0$ also follows
from that of $(a,\phi)$ by the same argument with a slight modification.

\begin{proof}[Proof of Lemma \ref{lem:ContHsofWKB}]
We first consider the case where $s<2$.
For simplicity, we denote $\psi(x)=e^{i\phi(x)}-1$.
Recall that $\Sobn{f}{s} \sim \Lebn{f}{2} + \hSobn{f}{s}$.
An elementary calculation provides
\begin{equation}\label{eq:L2apsi}
\begin{aligned}
	\Lebn{a \psi}{2} 
	\le{}& 2\norm{a}_{L^2(|x|\ge R)} \\
	&{} + \norm{a}_{L^2(|x|\le R)}
	 \(\sup_{|x|\le R}(e^{i\phi(x)}-e^{i\phi(0)}) +(e^{i\phi(0)}-1)\)\\
	\le{}&
	2\norm{a}_{L^2(|x|\ge R)}+ 
	\Lebn{a}{2} \(R \Lebn{\nabla\phi}\I +2\abs{\sin\frac{\phi(0)}{2}}\).
\end{aligned}
\end{equation}
The first term of the right hand side is small if $R$ is large.
Moreover, for any fixed (large) $R$, the second term is small if $\delta$ is sufficiently small.

We next estimate $\dot{H}^s$ norm of $a (e^{i\phi}-1)$.
Recall that, for $0<s<2$, 
\[
	\hSobn{a\psi}{s} 
	\sim\( \int_0^\I \( t^{-s} \sup_{|y|\le t} \norm{\delta_y(a\psi) - 2a\psi +\delta_{-y}(a\psi)}_{L^2(\R^2)} \)^2 \frac{dt}{t} \)^{\frac12},
\]
where $\delta_y$ is the shift operator, $(\delta_y f)(x):=f(x+y)$.
One easily verifies that
\begin{multline*}
	\int_1^\I \( t^{-s} \sup_{|y|\le t} \norm{\delta_y(a\psi) - 2a\psi +\delta_{-y}(a\psi)}_{L^2(\R^2)} \)^2 \frac{dt}{t} \\
	\le (4\norm{a\psi}_{L^2(\R^2)})^2 \int_1^\I t^{-1-2s}dt
	= \frac{8}{s} \norm{a\psi}_{L^2(\R^2)}^2.
\end{multline*}
We now consider the case $t\le 1$. A computation shows that
\begin{align*}
	\delta_y(a\psi) - 2a\psi +\delta_{-y}(a\psi)
	={}&(\delta_ya - 2a +\delta_{-y}a)\psi + (\delta_ya - \delta_{-y}a)
	(\psi -\delta_{-y}\psi) \\
	&{}+ \delta_y a [\delta_y\psi - 2\psi +\delta_{-y}\psi].
\end{align*}
The second term and the third term of the right hand side are estimated as
\begin{align*}
	\norm{(\delta_ya - \delta_{-y}a) (\psi -\delta_{-y}\psi)}_{L^2} \le 
	2|y|  \norm{\delta_ya - \delta_{-y}a}_{L^2} \Lebn{\nabla \phi}{\I}, \\
	\norm{\delta_y a [\delta_y\psi - 2\psi +\delta_{-y}\psi]}_{L^2}
	\le |y|^2 (\Lebn{a}2 \Lebn{\nabla\phi}\I^2 + 
	\Lebn{a}\I  \Lebn{\nabla^2 \phi}2),
\end{align*}
respectively. 
We next estimate the first term.
For $R \gg 1$, we have
\begin{align*}
	\Lebn{(\delta_ya - 2a +\delta_{-y}a)\psi }2
	\le{}& 2 \norm{\delta_ya - 2a +\delta_{-y}a}_{L^2(|x|\ge R)} \\
	&{}+ 4\((R+1)\Lebn{\nabla \phi}\I + 2\abs{\sin \frac{\phi(0)}2}\)\norm{a}_{L^2}.
\end{align*}
Let $\eta(x) \in C^\I(\R^n)$ be a function such that $0\le \eta \le 1$,
$\eta(x)=1$ for $|x|\ge1/2$, and $\eta(x)=0$ for $|x|\le 1/4$.
We put $\widetilde{a}_R(x)=a(x)\eta(x/R)$.
Then,
\[
	\norm{\delta_ya - 2a +\delta_{-y}a}_{L^2(|x|\ge R)}
	\le \norm{\delta_y\widetilde{a}_R - 2\widetilde{a}_R +\delta_{-y}\widetilde{a}_R}_{L^2(\R^n)}.
\]
Therefore, we conclude that
\begin{multline}\label{eq:dHsapsi}
	\int_0^1 \( t^{-s} \sup_{|y|\le t} \norm{\delta_y(a\psi) - 2a\psi +\delta_{-y}(a\psi)}_{L^2(\R^2)} \)^2 \frac{dt}{t} \\
	\le C\(\hSobn{\widetilde{a}_R}{s} + \( R\Lebn{\nabla \phi}\I + \abs{\sin \frac{\phi(0)}2}\)
	\norm{a}_{L^2}\) 
	\\+C\Lebn{\nabla \phi}\I  \hSobn{a}{s-1} + C(\Lebn{\nabla \phi}\I^2 \Lebn{a}2
	+\Lebn{\nabla^2 \phi}{2} \Lebn{a}\I).
\end{multline}
Plugging \eqref{eq:L2apsi} and \eqref{eq:dHsapsi}, we obtain
\begin{align*}
	\Sobn{a\psi}{s} \le {}&
	C\Sobn{\widetilde{a}_R}{s} +C \abs{\sin \frac{\phi(0)}2}
	\norm{a}_{L^2}\\ 
	&{} + C \Sobn{a}{s}( R\Lebn{\nabla \phi}\I
	+ \Lebn{\nabla \phi}\I^2
	+\Lebn{\nabla^2 \phi}{2}).
\end{align*}
For any $\eps>0$, we can choose $R$ so large that the first term of the 
right hand side is less than $\eps/3$.
Then, we can choose $\delta=\delta(\eps,R)$ such that
both the second term and the third term are less than $\eps/3$ if
$\Lebn{\nabla\phi}{\I} + \Lebn{\nabla^2\phi}2+|\phi(0)| < \delta$.

The case $s=2$ follows by direct calculations.

We show the case $s>2$ by induction.
We take positive integer $k$ and assume that the result is true for $k <s \le k+1$.
Then, we have
\begin{align*}
	\Sobn{a \psi}{s+1} \le{}& C  \Lebn{a \psi}{2}
	+ C \Sobn{\nabla (a\psi)}{s} \\
	\le {}& C  \Lebn{a \psi}{2}
	+  C \Sobn{(\nabla a)\psi}{s} + C \Sobn{ae^{i\phi}\nabla\phi}{s}.
\end{align*}
By \eqref{eq:L2apsi} and the assumption of the induction,
the first two terms of the right hand side are less than $\eps/3$ if $\delta $ is sufficiently small.
Now, since
\[
	\Sobn{ae^{i\phi}\nabla\phi}{s}
	\le C(\Lebn{\nabla\phi}\I \Sobn{ae^{i\phi}}{s} +
	\Sobn{\nabla^2 \phi}{s-1} \Lebn{ae^{i\phi}}\I ),
\]
the third term is also less than $\eps/3$ if $\delta$ is sufficiently small,
which completes the proof.
\end{proof}

\subsection{Uniqueness}
To complete the proof of Theorem \ref{thm:main1},
we show that the solution $u$ is unique.
It is important to note that the uniqueness of the system \eqref{eq:S1} does
not  directly means that of \eqref{eq:SP2}.
Namely, it implies no more than that the solution of \eqref{eq:SP2} which is written as
$u=ae^{i\phi}$ with a solution $(a,\phi)$ of \eqref{eq:S1}, is unique.
Then, what to show is that all solution of \eqref{eq:SP2} is written as
$u=ae^{i\phi}$ with a solution $(a,\phi)$ of \eqref{eq:S1}.
\begin{lemma}\label{lem:unique}
Let $s>2$ and define
\begin{align*}
	&A:=C([0,T);H^s(\R^2)) \cap C^1((0,T);H^{s-2}_{\mathrm{loc}}(\R^2))\\
	&B:=\left\{ \phi \in C^1([0,T)\times \R^2) ;
	\nabla \phi \in X^{s+1}(\R^2) \cap L^{2+}(\R^2) \right\}.
\end{align*}
Then, the following two statements are equivalent:
\begin{enumerate}
\item The system \eqref{eq:S1} has a unique solution $(a,\phi) \in A\times B$.
\item The system
\begin{equation}\label{eq:SP-HJ}
	\left\{
	\begin{aligned}
	&i \d_t u + \frac12 \Delta u = \l P u, \\
	&P =  -\frac{1}{2\pi}\int_{\R^2} \( \log \frac{|x-y|}{|y|} \) |u(y)|^2 dy , \\
	&\d_t \psi + \frac12 |\nabla \psi|^2 + \l P = 0, \\
	&u(0,x) = u_0(x), \quad \psi(0,x)=0
	\end{aligned}
	\right.
\end{equation}
has a unique solution $(u,\psi)\in A\times B$.
\end{enumerate}
\end{lemma}
By means of this lemma,
the uniqueness of \eqref{eq:SP2} is shown in the following way.
Set $A,B$ as in Lemma \ref{lem:unique}.
Let $u_1 ,u_2 \in A$ be two solutions of \eqref{eq:SP2}.
Then, we can solve the Hamilton-Jacobi equation
\[
	\d_t \psi + \frac12 |\nabla \psi|^2 + \l P = 0, \quad
	-\Delta P = |u|^2, \quad \psi(0,x)=0,
\]
in a way similar to the proof of Theorem \ref{thm:base}
and obtain $\psi_1 , \psi_2\in B$, respectively.
Note that $u$ is nothing but a source when we solve this equation.
Then, this lemma implies the solution of \eqref{eq:SP-HJ} is unique;
$(u_1,\psi_1)=(u_2,\psi_2)$.
In particular, $u_1=u_2$.
\begin{proof}[Proof of Lemma \ref{lem:unique}]
At first, we define mappings $f$ and $g$ by
\begin{align*}
	f:{}&A\times B \ni (a, \phi) \mapsto (ae^{i\phi}, \phi), \\
	g:{}&A\times B \ni (u, \psi) \mapsto (ue^{-i\psi}, \psi).
\end{align*}
By means of \eqref{eq:HsofWKB}, we see that the images of $f$ and $g$ are both 
subspaces of $A\times B$.
It is easy to verify that $f$ and $g$ are injective, and that $f \circ g = g\circ f = \mathrm{Id}$.
Therefore, both $f$ and $g$ are bijection from $A\times B$ to itself and $f^{-1}=g$.

Assume that $(a,\phi) \in A\times B$ is a unique solution of \eqref{eq:S1}.
Then, $(u,\psi)=f(a,\phi)$ solves \eqref{eq:SP-HJ}, and
this solution is unique since $f$ is bijective.
In the same way, if $(u,\psi)$ is a unique solution of \eqref{eq:SP-HJ},
then $(a,\phi)=f^{-1}(u,\psi)$ is a unique solution of \eqref{eq:S1}.
\end{proof}
\begin{remark}
In Lemma \ref{lem:unique}, the gauge invariance of the nonlinearity $Pu$,
that is, the property that $P$ depends only on the modulus $|u|$ and is independent of
the argument $u/|u|$, is fully employed.
By this property, it turns out that solutions of two Hamilton-Jacobi equations
in \eqref{eq:S1} and in \eqref{eq:SP-HJ} are identical.
\end{remark}

\section{Proof of Theorem \ref{thm:main2}}\label{sec:proof2}
We see in Theorem \ref{thm:base} that the system \eqref{eq:S} has 
a (unique) solution $(a^\eps, \phi^\eps)$ and that it converges to
$(a_0,\phi_0)$ solving \eqref{eq:S0} if $A_0=\lim_{\eps\to0} A^\eps$ exists.
Then, one verifies that
$(b^\eps,\psi^\eps)=((a^\eps-a_0)/\eps,(\phi^\eps-\phi_0)/\eps)$
solves a system similar to \eqref{eq:S}.
Thus, mimicking the proof of Theorem \ref{thm:base},
we can prove that $(b^\eps,\psi^\eps)$ exists and uniformly bounded if 
\[
	b^\eps(0) = \frac{A^\eps - A_0}{\eps}
\]
is uniformly bounded.
As a result, we obtain the following. For the details of the proof,
consult \cite{CM-AA,Grenier98}.
\begin{assumption}\label{asmp:expansion}
Let $N_0 \ge 1$ and assume the following for some $s>3+2N_0$:
\begin{itemize}
\item $A^\eps \in H^s(\R^2)$  and $\norm{A^\eps}_{H^{s}(\R^2)}$
is uniformly bounded.
Moreover, $A^\eps$ is expanded as
\[
	A^\eps = A_0 + \eps A_1 + \cdots + \eps^{N_0} A_{N_0} + o(\eps^{N_0})
	\IN H^s(\R^2).
\]
\item $\Phi \in C^1([0,T)\times \R^2)$ with $\nabla \Phi \in X^{s+1}(\R^2)\cap L^p
(\R^2)$ for some $p\in (2,\I)$.
\end{itemize}
\end{assumption}
\begin{proposition}\label{prop:expansion}
Let Assumption \ref{asmp:expansion} be satisfied.
Then, the unique solution $(a^\eps,\phi^\eps)$ of \eqref{eq:S} 
given by Theorem \ref{thm:base} has the following expansion:
\begin{align*}
	a^\eps &{}= a_0 + \eps a_1 + \cdots + \eps^{N_0} a_{N_0} + o(\eps^{N_0}), 
	\IN L^\I([0,T);H^{s-2N_0})\\
	\phi^\eps &{}= \phi_0 + \eps \phi_1 + \cdots + \eps^{N_0} \phi_{N_0} + o(\eps^{N_0}), \IN L^\I([0,T);L^\I_{\mathrm{loc}}), \\
	\nabla \phi^\eps &{}= \nabla\phi_0 + \cdots + \eps^{N_0}
	\nabla\phi_{N_0} + o(\eps^{N_0}), \IN L^\I([0,T);X^{s+1-2N_0}\cap L^{2+}),
\end{align*}
where, for all $j\in [0,N_0]$, $a_j\in C([0,T);H^{s-2j})$ and $\phi_j \in C^1([0,T)\times \R^2)$ with $\nabla \phi_j \in X^{s+1-2j}\cap L^{2+}$.
\end{proposition}
At this stage, we see that the WKB approximation of the solution
holds on any bounded domain; there exist $\phi_0$ and $\beta_j$ such that
\[
 	u^\eps = e^{i\frac{\phi^\eps}{\eps}}(\beta_0 + \cdots
 	+ \eps^{N_0-1} \beta_{N_0-1} + o(\eps^{N_0-1})) \IN L^\I([0,T);H^{s-2N_0}_{\mathrm{loc}}(\R^2)).
\]
To show the approximation in Theorem \ref{thm:main2}
which is valid not on a bounded domain but on $\R^2$, 
we prepare the following two lemma.
\begin{lemma}\label{lem:weightL2}
Let $w$ be a real-valued function of $x\in \R^2$
such that $\nabla w \in L^\I$.
For a solution $u^\eps\in C((-T,T);H^1(\R^2))$ of \eqref{eq:SP2},
it holds that
\[
	\frac{d}{dt} \int w |u^\eps(t)|^2 dx
	=  \eps\Im \int (\nabla w\cdot \nabla u^\eps(t)) \overline{u^\eps(t)} dx.
\]
For a solution $(a^\eps,\phi^\eps)$ of \eqref{eq:S},
it holds that
\[
	\frac{d}{dt} \int w |a^\eps(t)|^2 dx
	=  \eps\Im \int (\nabla w\cdot \nabla a^\eps(t)) \overline{a^\eps(t)} dx+
	\int (\nabla w\cdot \nabla \phi^\eps(t)) |a^\eps(t)|^2 dx.
\]
\end{lemma}
\begin{proof}
The first identity  follows from
\begin{align*}
	\frac{d}{dt} \int w |u^\eps(t)|^2 dx
	&{}= 2\Re \int w \d_t u^\eps(t) \overline{u^\eps(t)} dx\\
	&{}= - \eps\Im \int w \Delta u^\eps(t) \overline{u^\eps(t)} dx \\
	&{}=  \eps\Im \int (\nabla w\cdot \nabla u^\eps(t)) \overline{u^\eps(t)} dx,
\end{align*}
and so does the second one from this identity and $u^\eps=a^\eps e^{i\phi^\eps/\eps}$.
\end{proof}

The next lemma is the key for the proof.
\begin{lemma}\label{lem:weight}
Let $N\ge1$ be an integer and let 
Assumption \ref{asmp:expansion} be satisfied for some $N_0=2N$.
Let $a_j$ ($j\in[0,2N]$) be given in Proposition \ref{prop:expansion}.
Let $\alpha\in (0,1]$. If
\begin{equation}\label{eq:higherdecay}
	(1+|x|)^{\frac{\alpha}{2^{j}}}|a_j(t)|^2 \in L^1(\R^2), \qquad j=0,1,\dots, N
\end{equation}
holds at the initial time $t=0$, then \eqref{eq:higherdecay} holds for all $t\in [0,T)$.
\end{lemma}

\begin{proof}
We show \eqref{eq:higherdecay} by induction on $j$.

{\bf Step 1}.
We first consider $j=0$.
From Lemma \ref{lem:weightL2}, we have
\begin{align*}
	\frac{d}{dt} \int (1+|x|)^{\alpha} |a^\eps(t)|^2 dx
	={}&  \alpha \eps\Im \int \( \frac{x}{|x|(1+|x|)^{\alpha-1}}\cdot \nabla a^\eps(t)\) \overline{a^\eps(t)} dx\\
	&{}+\alpha\int \(\frac{x}{|x|(1+|x|)^{\alpha-1}}\cdot \nabla \phi^\eps(t)\) |a^\eps(t)|^2 dx.
\end{align*}
Let $\eps=0$ to obtain
\begin{align*}
	\int (1+|x|)^{\alpha} |a_0(t)|^2 dx
	={}& \int (1+|x|)^{\alpha} |A_0|^2 dx \\
	&{}+\alpha \int_0^t \int \(\frac{x}{|x|(1+|x|)^{\alpha-1}}\cdot \nabla \phi_0(s)\) |a_0(s)|^2 dx\,ds\\
	\le{}&\int (1+|x|)^{\alpha} |A_0|^2 dx \\
	&{}+\alpha t \norm{\nabla \phi_0}_{L^\I([0,t]\times \R^2)}
	\norm{a_0}_{L^\I([0,t],L^2)}\\
	<{}& \I.
\end{align*}

{\bf Step 2}.
We now assume for induction that \eqref{eq:higherdecay} holds for $j=0,1,\cdots,k-1$
($k\le N$) and show \eqref{eq:higherdecay} for $j=k$.
Comparing the $\eps^{2k}$-order term of the both sides of
\begin{align*}
	\frac{d}{dt} \int (1+|x|)^{\frac{\alpha}{2^k}} |a^\eps(t)|^2 dx
	={}&  \frac{\alpha}{2^k} \eps\Im \int \( \frac{x}{|x|(1+|x|)^{1-\frac{\alpha}{2^k}}}\cdot \nabla a^\eps(t)\) \overline{a^\eps(t)} dx\\
	&{}+\frac{\alpha}{2^k}\int \(\frac{x}{|x|(1+|x|)^{1-\frac{\alpha}{2^k}}}\cdot \nabla \phi^\eps(t)\) |a^\eps(t)|^2 dx,
\end{align*}
we deduce
\begin{multline*}
	\frac{d}{dt} \int (1+|x|)^{\frac{\alpha}{2^k}} \(|a_k(t)|^2 + \sum_{l=0}^{k-1}2\Re (a_{2k-l}(t)\overline{a_l(t)})\)dx \\
	= \frac{\alpha}{2^k} \Im \int \sum_{l=0}^{2k-1} \( \frac{x}{|x|(1+|x|)^{1-\frac{\alpha}{2^k}}}\cdot \nabla a_{2k-1-l}(t)\) \overline{a_{l}(t)} dx\\
	+\frac{\alpha}{2^k}\int \sum_{l_1+l_2+l_3=2k}\(\frac{x}{|x|(1+|x|)^{1-\frac{\alpha}{2^k}}}\cdot \nabla \phi_{l_1}(t)\) a_{l_2}(t)\overline{a_{l_3}(t)} dx.
\end{multline*}
Denote the right hand side by $e(t)$.
The weight function $x/|x|(1+|x|)^{\alpha/2^k-1}$ on the right side is bounded uniformly in $x$,
and so $|\int_0^t e(s)ds|<\I$ follows from the assumption and the H\"older inequality.
Since
\begin{align*}
		&\int (1+|x|)^{\frac{\alpha}{2^k}} 2\Re (a_{2k-l}(t)\overline{a_l(t)}) dx\\
		&{}\ge - 2\norm{a_{2k-l}(t)}_{L^2(\R^2)}
	\( \int (1+|x|)^{\frac{\alpha}{2^{k-1}}}|a_{l}(t)|^2 dx\)^{\frac12} \\
	&{}\ge - 2\norm{a_{2k-l}(t)}_{L^2(\R^2)}
	\( \int (1+|x|)^{\frac{\alpha}{2^{l}}}|a_{l}(t)|^2 dx\)^{\frac12}
	>-\I
\end{align*}
holds for $l=0,\cdots, k-1$ 
from the H\"older inequality and assumption of induction, we conclude that
\[
	\int (1+|x|)^{\frac{\alpha}{2^k}} |a_k(t)|^2 dx < \I,
\]
which completes the proof.
\end{proof}
\begin{proof}[Proof of Theorem \ref{thm:main2}]
Notice that the Assumption \ref{asmp:main2} implies that
Assumption \ref{asmp:expansion} is filled for $N_0=2N(>N)$.
Therefore, we have a unique solution $(a^\eps,\phi^\eps)$ of \eqref{eq:S}
and its expansion
\begin{equation}\label{eq:mainexpansion}
\begin{aligned}
	a^\eps &{}= a_0 + \eps a_1 + \cdots + \eps^{N} a_{N} + o(\eps^{N}), 
	\IN L^\I([0,T);H^{s-2N})\\
	\phi^\eps &{}= \phi_0 + \eps \phi_1 + \cdots + \eps^{N} \phi_{N} + o(\eps^{N}), \IN L^\I([0,T);L^\I_{\mathrm{loc}}), \\
	\nabla \phi^\eps &{}= \nabla\phi_0 + \cdots + \eps^{N}
	\nabla\phi_{N} + o(\eps^{N}), \IN L^\I([0,T);X^{s+1-2N}\cap L^{2+})
\end{aligned}
\end{equation}
by Proposition \ref{prop:expansion}.
Moreover, assumption of Lemma \ref{lem:weight} is also satisfied
and so \eqref{eq:higherdecay} holds for $j=1,2,\cdots,N$.
By the Taylor expansion, we have
\[
	\abs{e^{i\eps\phi_1}-\sum_{l_1=0}^N \eps^{l_1} \frac{(i\phi_1)^{l_1}}{l_1!}}
	\le \frac{\eps^{N+1}|\phi_1|^{N+1} }{(N+1)!}
\]
Recall that $|\phi_1(x)|=O(\log|x|)$ as $|x|\to\I$,
which gives
\[
	\abs{a_0e^{i\eps\phi_1}-\sum_{l_1=0}^N \eps^{l_1} a_0\frac{(i\phi_1)^{l_1}}{l_1!}}
	\le C\eps^{N+1} (1+\log\Jbr{x})^{N+1} |a_0| \in L^2
\]
together with \eqref{eq:higherdecay}. Thus,
\[
	a_0e^{i\eps\phi_2} = \sum_{l_1=0}^N \eps^{l_1} a_0\frac{(i\phi_2)^{l_1}}{l_1!} + o(\eps^N)\IN L^\I([0,T);L^2).
\]
Since a similar expansion holds for all term of the form
\[
	\eps^{k_1}a_{k_1} e^{i\eps^{k_2-1}\phi_{k_2}},
\]
combining the expansions \eqref{eq:mainexpansion}
and $u^\eps=a^\eps e^{i\phi^\eps/\eps}$, we conclude that
\begin{equation}\label{eq:goalWKB}
\begin{aligned}
	u^\eps e^{-i\frac{\phi_0}\eps} &{}= 
	e^{i\phi_1} e^{i\eps\phi_2} \cdots e^{i\eps^{N-1}\phi_N} e^{o(\eps^{N-1})}
	(a_0 + \cdots + \eps^N a^N + o(\eps^N)) \\
	&{} = \beta_0 + \eps \beta_1 + \cdots + \eps^{N-1}\beta_{N-1} + o(\eps^{N-1})	
\end{aligned}
\end{equation}
in $L^\I([0,T);L^2(\R^2))$, where $\beta_0=a_0e^{i\phi_1}$ and 
$\beta_j$ ($j\ge1$) is given by the following way:
For a positive integer $l$, we call a multi-index $\sigma \in (\N \cup \{0\})^l$
is a weighted partition of $l$ if $\sum_{k=1}^l k \sigma_k = l$.
The function $\beta_j$ ($j\ge1$) in \eqref{eq:goalWKB} is given explicitly as
\[
	\beta_j = e^{i\phi_1}\(a_j + \sum_{l=1}^j a_{j-l}\sum_{\sigma:\text{weigted partition of }l} \prod_{k=1}^l\frac{i^{\sigma_k}(\phi_{k+1})^{\sigma_k}}{\sigma_k!}\).
\]
Note that $\beta_j\in C([0,T);L^2(\R^2))$ follows from \eqref{eq:higherdecay}.
\end{proof}
\begin{remark}
The feature of the two-dimensional case is that
not only $\phi_0$ but also all of $\phi_j$ ($j\ge1$) may grow at the spatial infinity
though they are identically zero at the initial time.
This growth comes from Poisson terms (see \eqref{eq:log}).
This is why amplitudes are required to be in some weighted $L^2$ space. 
\end{remark}
\appendix

\section{Poisson equation in the two dimensional whole space}\label{sec:Poisson}
In this appendix, we consider the Poisson equation
\begin{equation}\label{eq:P}
	-\Delta P = f \IN \R^2
\end{equation}
with the conditions
\begin{gather}\label{eq:C1a}
	|\nabla P|\to 0 \text{ as } |x|\to \I, \quad P(0)=0, \\
\label{eq:C1b}
	\nabla P \in L^\I(\R^2).
\end{gather}
We briefly recall the higher dimension case $n\ge3$.
It is well known that the solution $P$ is defined by
the Fourier transform or by the Newtonian potential as
\begin{align}\label{eq:solP3a}
	P(x) ={}& \F^{-1} \left[ \frac{1}{|\xi|^2} \F f(\xi)\right](x) \\
		={}& \frac{1}{n(n-2)\omega_n} (|x|^{2-n}*f)(x), \label{eq:solP3b}
\end{align}
where $\omega_n$ denotes the volume of the unit sphere in $\R^n$.
In this case, it can be said that \eqref{eq:P} in $\R^n$
is posed with the condition
\begin{equation}\label{eq:C2}
	P\to 0 \text{ as } |x|\to \I, \quad  P \in L^\I(\R^n).
\end{equation}
For a good $f$, say $f \in \mathcal{S}(\R^n)$, the solution $P$ defined by
\eqref{eq:solP3a} or \eqref{eq:solP3b} satisfies \eqref{eq:C2},
and  is unique by Liouville's theorem.

In the two dimensional case, it is not possible to define the solution
by \eqref{eq:solP3a} in general (even in the distribution sense)
because of the singularity of $|\xi|^{-2}$.
In \cite{AN-MMAS,ZhSIAM,JW-CPDE}, the definition \eqref{eq:solP3a} is 
employed under several assumption on $f$ which 
provides $\F f(\xi)=O(|\xi|)$ as $\xi\to0$.
To realize it, it is almost necessary to suppose the following neutrality condition:
\begin{equation}\label{eq:A:neutrality}
	2\pi\F f(0) = \int_{\R^2} f(x) dx = 0.
\end{equation}
This condition is, however, very restrictive in some case.
For example, in our systems \eqref{eq:SP} or \eqref{eq:SPe},
this condition excludes all nontrivial solutions.
To avoid such a situation, we observe the fact that
\begin{equation}\label{eq:nablaP}
	\F^{-1} \left[ \frac{-i\xi}{|\xi|^2} \F f(\xi)\right](x)
\end{equation}
(which may be equal to $\nabla P$) is well-defined even in the two-dimensional
case, and we modify the condition \eqref{eq:C2} into \eqref{eq:C1a}--\eqref{eq:C1b},
so that the Poisson equation \eqref{eq:P} has a solution.
The idea is the following: 
If the gradient of $P$ was defined uniquely, then $P$ should be given uniquely
by the line integral of it under $P(0)=0$.

We denote $p^*=2p/(2-p)$ for $p<2$.
$p^*$ is increasing in $p$, and $1^*=2$.
\begin{theorem}\label{thm:2dPoisson}
\begin{itemize}
\item If $f \in L^{p_0}(\R^2)$ for some $p_0 \in (1,2)$, then
\begin{equation}\label{eq:formulaP}
	P(x) = -\frac{1}{2\pi} \int_{\R^2} \(\log\frac{|x-y|}{|y|}\) f(y) dy
\end{equation}
is well-defined and is a weak solution of \eqref{eq:P} in such a sense that
its weak derivative
\begin{equation}\label{eq:formulanP}
	\nabla P(x) = -\frac{1}{2\pi} \int_{\R^2} \frac{x-y}{|x-y|} f(y) dy
	\in L^{p_0^*}(\R^2)
\end{equation}
satisfies $\Jbr{\nabla P, \nabla \varphi}=-\Jbr{f,\varphi}$ for all $\varphi \in \mathcal{S}(\R^2)$.
Moreover, this solution satisfies \eqref{eq:C1a} and if $f\in L^1(\R^2)$ then
\begin{equation}\label{eq:log}
	\limsup_{|x|\to\I}\frac{|P(x)|}{\log\Jbr{x}} \le \frac{\Lebn{f}1}{2\pi}.
\end{equation}
\item If, in addition, $f$ is continuous and $\nabla f \in L^{q_0}(\mathbb{R}^2)$ 
for some $q_0>2$, then $P$ is in $C^2(\R^2)$ and is the unique classical solution of \eqref{eq:P}
with \eqref{eq:C1a}--\eqref{eq:C1b}.
Moreover, $P$ satisfies 
$\nabla P \in L^r(\R^2)$ for $r \in [p_0^*,\I]$, $\nabla^2 P \in L^p(\R^2)$
for $p\in [p_0,\I]$, and $\nabla^3 P \in L^{q_0}(\R^2)$.
\end{itemize}
\end{theorem}
\begin{remark}
The operator $\nabla (-\Delta)^{-1}:=-\F^{-1}i\xi/|\xi|^2\F$ is 
defined as a bounded operator from $L^{p_0}(\R^2)$ to $L^{p_0^*}(\R^2)$
for $p_0\in(1,2)$.
Remark that both \eqref{eq:nablaP} and \eqref{eq:formulaP} make sense
for $f\in L^{p_0}(\R^2)$, $p_0\in(1,2)$.
Therefore, it can be said that
\eqref{eq:formulaP} is one of the ``proper'' integral of \eqref{eq:nablaP}.
Remark that, from this point of view, the Newtonian potential
 $-(2\pi)^{-1}(\log|x|*f)$ is not proper (see Proposition \ref{prop:convolution}
and the consequent remarks, below).
\end{remark}
\begin{remark}\label{rmk:nPL2}
Note that $\nabla P \in L^2(\R^2)$ only if $f$ satisfies the neutrality condition
$-2\pi\F f(0)=\int_{\R^2} f dx= 0$.
This is because $\Lebn{\nabla P}2 = \Lebn{|\xi|^{-1}\F f}2$.
\end{remark}
\begin{proof}
Recall that $\log|x| \in L^p_{\mathrm{loc}}(\R^2)$ for all $1\le p <\I$ and
$\log(|x-y|/|y|)=O(|y|^{-1})$ as $|y|\to\I$.
Therefore, it follows that
$\log(|x-y|/|y|) \in L^{p_0/(p_0-1)}_y(\R^2)$ for any fixed $x\in \R^2$,
and so that $P$ is well-defined for $f\in L^{p_0}(\R^2)$ by the H\"older inequality .
One easily verifies that the weak derivative of $P$ is given by \eqref{eq:formula}.
By the Hardy-Littlewood-Sobolev inequality, \eqref{eq:formulanP} is also well-defined
for $f\in L^{p_0}(\R^2)$ and belongs to $L^{p_0^*}(\R^2)$.
Thus, \eqref{eq:C1a} is satisfied.
A computation shows that $\F(x/|x|^2) = -i\xi/|\xi|^2$.
Therefore, we see that
$\nabla P= -\F^{-1}[(i\xi/|\xi|^2)\F f]$ solves \eqref{eq:P} in the distribution sense.
To complete the proof of the former part, we show \eqref{eq:log}.
Set $K(x,y) := \log \frac{|x-y|}{\Jbr{y}}$
and
\begin{align*}
	K_1(x,y;\delta)&{}:={\bf 1}_{\{|x-y|\ge\delta\}}K(x,y), &
	K_2(x,y;\delta)&{}:={\bf 1}_{\{|x-y|\le\delta\}}K(x,y),
\end{align*}
where $\Jbr{y}=\sqrt{1+|y|^2}$
and $\delta\in (0,1]$ to be chosen later.
Let us first show that
\begin{equation}\label{eq:estK1}
	\sup_{y} \abs{K_1(x,y;\delta)} \le \log\Jbr{x} + \log \sqrt3 + \log\frac1\delta,
	\qquad \forall x \in \R^2.
\end{equation}
Put $x-y=-w$. Notice that the support of $K_1$ is written as $\{|w|\ge \delta \}$.
By triangle inequality, we obtain
\[
	\log \frac{|w|}{\sqrt{1+(|w|+|x|)^2}} \le 
	\log \frac{|w|}{\Jbr{w+x}} \le
	\log \frac{|w|}{\sqrt{1+(|w|-|x|)^2}}.
\]
The left hand side is always negative and monotone increasing in $|w|(\ge \delta)$,
and so we have the following bound:
\begin{align*}
	\abs{\log\frac{|w|}{\sqrt{1+(|w|+|x|)^2}}} \le {}&
	\abs{\log\frac{\delta}{\sqrt{1+(\delta+|x|)^2}}} \\
	\le {}& \log \sqrt3 + \log \Jbr{|x|} + \log \frac1\delta,
\end{align*}
where we have used the relation $1\le 1+(\delta+|x|)^2 \le 3(1+|x|^2)$ 
for $\delta \in (0,1]$.
On the other hand, the right hand side is, in $|w|$,
increasing if $\delta \le |w| \le |x|+1/|x|$
and decreasing if $|w|\ge |x|+1/|x|$, and tends to zero as $|w|\to\I$.
Therefore,
\[
	\abs{\log \frac{|w|}{\sqrt{1+(|w|-|x|)^2}}} \le
	\max \( \log\Jbr{x} , - \log \frac{\delta}{\sqrt{1+(\delta-|x|)^2}}  \).
\]
Then, to show \eqref{eq:estK1}, it suffices to note that
\[
	-\log \frac{\delta}{\sqrt{1+(\delta-|x|)^2}}  \le \log\sqrt2 + \log\Jbr{|x|}
	+\log \frac1\delta.
\]
It follows from \eqref{eq:estK1} that
\begin{equation*}
	\frac{\abs{\int_{\R^2} K_1(x,y;\delta)f(y)dy}}{\log\Jbr{x}}
	\le \Lebn{f}1 + \frac{\Lebn{f}1(\log\sqrt3 + \log(1/\delta))}{\log \Jbr{x}}.
\end{equation*}
On the other hand, applying the inequality$\Jbr{|x+w|} \le \sqrt3 \Jbr{x}$
for $|w|\le \delta \le 1$, we obtain
\begin{align*}
	\norm{K_2(x,\cdot;\delta)}_{L^q_y} 
	&\le \norm{\log\Jbr{w+x}}_{L^{q}(|w|\le\delta)} + \norm{\log|w|}_{L^{q}(|w|\le\delta)}, \\
	&\le (\pi \delta^2)^{\frac1q} (\log\Jbr{x}+\log\sqrt3) + \norm{\log|w|}_{L^{q}(|w|\le\delta)},
\end{align*}
for $q=p_0/(p_0-1)$, which yields
\begin{equation}\label{eq:estK2}
\begin{aligned}
	\frac{\abs{\int_{\R^2} K_2(x,y;\delta)f(y)dy}}{\log\Jbr{x}}
	\le{}& (\pi\delta^2)^{\frac1q}\Lebn{f}{p_0}\\
	&{}+ \Lebn{f}{p_0}\frac{(\pi\delta^2)^{\frac1q}\log\sqrt3 + \norm{\log|w|}_{L^{q}(|w|\le\delta)}}{\log\Jbr{x}}.
\end{aligned}
\end{equation}
Thus, we let $\delta=(\log\Jbr{x})^{-1}$ to conclude from \eqref{eq:estK1} and \eqref{eq:estK2} that
\begin{align*}
	\frac{|P(x)|}{\log\Jbr{x}}  \le {}&
	\frac{\abs{\int_{\R^2} K_1(x,y;(\log\Jbr{x})^{-1})f(y)dy}}{2\pi\log\Jbr{x}} \\
	&{}+ \frac{\abs{\int_{\R^2} K_2(x,y;(\log\Jbr{x})^{-1})f(y)dy}}{2\pi\log\Jbr{x}}
	+ \frac{|\int_{\R^2} \log (\frac{\Jbr{y}}{|y|})f(y)dy|}{2\pi\log\Jbr{x}} \\
	\to {}& \frac{\Lebn{f}1}{2\pi}
\end{align*}
as $|x|\to\I$, where we have used the fact that $\log ({\Jbr{y}}/{|y|}) \in L^{p_0/(p_0-1)}(\R^2)$.
\smallbreak

Let us proceed to the proof of the second part of the theorem.
Note that, for all $j,k,l \in \{1,2\}$, it holds that
\[
	\d_j\d_k P = R_j R_k f, \quad \d_j \d_k \d_l P = R_jR_k\d_l f ,
\]
where $R_j$ denotes the Riesz transform $\F^{-1} (-i\xi_j/|\xi|)\F$.
Applying the $L^p$-boundedness of the Riesz transform ($1<p<\I$),
we see that
\[
	\d_j\d_k P \in L^p(\R^2), \, \forall p \in [p_0,\I) ,\quad
	\d_j \d_k \d_l P \in L^{q_0}(\R^2).
\]
Then, $\d_j\d_k P \in L^\I(\R^2)$ also follows the Gagliardo-Nirenberg inequality.
The Hardy-Littlewood-Sobolev yields $\nabla P \in L^r(\R^2)$ for $r\in [q_0,\I)$.
Finally, a use of H\"older inequality shows $\nabla P\in L^\I(\R^2)$.
By the continuity argument, we conclude that $P \in C^2(\R^2)$.

We finally prove the uniqueness of the classical solution.
Let $P_1,P_2$ be two solutions of \eqref{eq:P} with \eqref{eq:C1a}--\eqref{eq:C1b}.
Then, $w:=P_1-P_2$ is a harmonic function.
Differentiating $\Delta w=0$ in $x_1$,
we see that $\d_1 w$ is also a harmonic function on $\R^2$.
Since we have already known that $\d_1 w$ is bounded,
$\d_1 w$ is a constant.
However, $\d_1 w \to 0$ as $|x|\to\I$ and so $\d_1 w \equiv 0$.
Similarly, $\d_2 w \equiv 0$.
Therefore, $w$ is a constant and so $w(x) \equiv w(0)=0$,
which shows $P_1 \equiv P_2$.
\end{proof}
\subsection{A solution given by the Newtonian potential}
We can also give a rigorous meaning of the Newtonian potential
\begin{equation}\label{eq:2Dconvolution}
	\widetilde{P}(x)=-\frac{1}{2\pi}\int_{\R^2} (\log|x-y|) f(y)dy
\end{equation}
as a solution of the Poisson equation.
Notice that $-\frac{1}{2\pi}\log|x|$ is the Newtonian kernel in two dimensions and
so that $\widetilde{P}$ is a two-dimensional version of \eqref{eq:solP3b}.
\begin{proposition}\label{prop:convolution}
Let $f \in L^{p_0}(\R^2)$ for some $p_0 \in (1,2)$ and let $\widetilde{P}$ be as
in \eqref{eq:2Dconvolution}.
If $\widetilde{P}(x)$ is finite at some $x \in \R^2$,
then it is finite for all $x\in \R^2$ and,
moreover, $\widetilde{P}(x)=P(x)+\widetilde{P}(0)$,
where $P$ is the solution of \eqref{eq:P} with \eqref{eq:C1a}--\eqref{eq:C1b}
given by Theorem \ref{thm:2dPoisson}.
\end{proposition}
Notice that the proposition implies the following:
\begin{itemize}
\item If $\widetilde{P}(x)$ diverges at some $x\in \R^2$ then
it necessarily diverges for all $x\in \R^2$ under the same assumption on $f$.
\item The difference between $P$ and $\widetilde{P}$ is merely
a constant $\widetilde{P}(0)$.
However, when we consider $\widetilde{P}$, 
we need an additional assumption on $f$
only for saying that this constant is finite.
\item If $f\in L^{p_0}(\R^2)$ is so that $\widetilde{P}(0)$ is finite,
then $\widetilde{P}$ is a weak solution of \eqref{eq:P} with 
the condition $P(0)=\widetilde{P}(0)$ and $|\nabla P| \to 0$ as $|x|\to\I$.
\end{itemize}
The proof of this proposition is obvious: It suffices to mention that, for any $f \in L^{p_0}(\R^2)$ ($p_0\in(1,2)$) and
$x_1$, $x_2 \in \R^2$,
\[
	-\frac{1}{2\pi}\int_{\R^2} \(\log\frac{|x_1-y|}{|x_2-y|}\) f(y)dy
\]
is finite.

\subsection*{Acknowledgments}
The author expresses his deep gratitude
to Professors Yoshio Tsutsumi and Remi Carles for fruitful discussions.
Deep appreciation goes to Professor Hideo Kubo 
for his valuable advice and constant encouragement.
This research is supported by JSPS fellow.

\bibliographystyle{amsplain}
\bibliography{caustic}

\end{document}